\documentclass{siamart1116}

\newtheorem{prop}[theorem]{Proposition}
\usepackage[normalem]{ulem}
\usepackage{csquotes}
\usepackage{tikz}
\usepackage{makecell}
\usetikzlibrary{arrows,decorations.markings}
\usepackage{pgf}
\usepackage{color}
\usepackage{amsmath}
\usepackage{graphicx}
\usepackage{amsfonts}
\usepackage{amssymb}%
\usepackage{latexsym}
\usepackage{psfrag}
\usepackage{accents}
\usepackage{hyperref}
\definecolor{myblue}{rgb}{0,0,0.6}
\hypersetup{colorlinks=true,
linkcolor=myblue,citecolor=myblue,filecolor=myblue,urlcolor=myblue}

\newcommand{\ri}{{\mathrm{i}}}
\newcommand{\re}{{\mathrm{e}}}
\newcommand{\erf}{\mathrm{erf}}
\newcommand{\erfc}{\mathrm{erfc}}
\newcommand{\erfcx}{\mathrm{erfcx}}

\newcommand{\R}{\mathbb{R}}
\newcommand{\N}{\mathbb{N}}
\newcommand{\Z}{\mathbb{Z}}
\newcommand{\C}{\mathbb{C}}
\numberwithin{theorem}{section}

\newcommand{\TheTitle}{Computation of the Complex Error Function using Modified Trapezoidal Rules}
\newcommand{\TheAuthors}{Mohammad Al Azah, Simon N. Chandler-Wilde}

\headers{Computation of the Error Function using Trapezoidal Rules}{\TheAuthors}

\title{{\TheTitle}}

\author{
  Mohammad Al Azah\thanks{School of Social and Basic Sciences, Al Hussein Technical University (HTU), Amman, Jordan
    (\email{mohammad.azah@htu.edu.jo}).}
  \and
  Simon N. Chandler-Wilde\thanks{Department of Mathematics and Statistics, University of Reading, Whiteknights, PO Box 220, Reading RG6 6AX, UK (\email{S.N.Chandler-Wilde@reading.ac.uk}).}
}

\ifpdf
\hypersetup{
  pdftitle={\TheTitle},
  pdfauthor={\TheAuthors}
}
\fi

\begin{document}

\maketitle

\begin{abstract}
In this paper we propose a method for computing
the Faddeeva function $w(z) := \re^{-z^2}\erfc(-\ri\,z)$ via truncated modified trapezoidal rule approximations to integrals on the real line. Our starting point is the method due to Matta and Reichel ({\em Math.\ Comp.} {\bf 25} (1971), pp.~339--344) and Hunter and Regan ({\em Math.\ Comp.} {\bf 26} (1972), pp.~339--541). Addressing shortcomings flagged by  Weideman ({\em SIAM.\ J.\ Numer.\ Anal. } {\bf 31} (1994), pp.~1497--1518), we construct approximations which we prove are exponentially convergent as a function of $N+1$, the number of quadrature points, obtaining error bounds which show that accuracies of $2\times 10^{-15}$ in the computation of $w(z)$ throughout the complex plane are achieved with $N = 11$, this confirmed by computations. These approximations, moreover, provably achieve small relative errors throughout the upper complex half-plane where $w(z)$ is non-zero. Numerical tests suggest  that this new method is competitive, in accuracy and computation times, with existing methods for computing $w(z)$ for complex $z$.
\end{abstract}



\section{Introduction}
The  complementary error function is defined by \cite[(7.2.2)]{DL10}
\begin{equation}\label{erfc_dfn}
\erfc(z)=\frac{2}{\sqrt{\pi}}\int_{z}^{\infty}\re^{-t^2}\,dt,
\end{equation}
for all $z=x+\ri y$ ($x,y\in \R$).
This paper is concerned with approximating $\erfc(z)$ through approximating an integral representation for the related Faddeeva function, defined by \cite[(7.2.3)]{DL10}
\begin{equation} \label{wdef}
w(z) := \re^{-z^2}\erfc(-\ri\,z).
\end{equation}
It is well known \cite[(7.1.4)]{AS64} that 
\begin{equation}\label{wint}
w(z) = \frac{\ri }{\pi} \int_{-\infty}^\infty  \frac{\re^{-t^2}}{z-t} \, d t = \frac{\ri z}{\pi} \int_{-\infty}^\infty  \frac{\re^{-t^2}}{z^2-t^2} \, d t,\quad\mathrm{Im}(z)>0,
\end{equation}
and this is our starting point. It is sufficient to devise methods to compute $w(z)$ for $z$ in the first quadrant since values in the other quadrants can be obtained using the symmetries \cite[(3.1) and (3.2)]{PW90}
\begin{equation} \label{symm}
w(-z)=2\re^{-z^2}-w(z)\quad\mbox{and}\quad w(\overline{z})=\overline{w(-z)}.
\end{equation}

It follows from \eqref{wdef}--\eqref{wint} that
\begin{equation} \label{erfc_int}
\erfc(z) =\frac{z\,\re^{-z^2}}{\pi} \int_{-\infty}^\infty  \frac{\re^{-t^2}}{z^2+t^2} \, d t,\quad x = \mathrm{Re}(z)>0.
\end{equation}
Starting from this integral representation Chiarella and Reichel \cite{CR68} and Matta and Reichel \cite{MR71} showed, by the contour integration arguments that we recall in \S\ref{sec:BasicHistory}, that, for $x>0$,
\begin{equation}\label{I*(h,0)}
\erfc(z)=\frac{hz\re^{-z^2}}{\pi}\sum_{k=-\infty}^{\infty}\frac{\re^{-k^2 h^2}}{z^2+k^2 h^2}+\frac{2\,H\!(\pi/h-x)}{1-\re^{2\pi z/h}} + E(h).
\end{equation}
Here the first term is the trapezoidal rule approximation to \eqref{erfc_int} using a step size $h>0$, the second is a modification that arises from Cauchy's residue theorem (expressed using the standard Heaviside step function $H$, with $H(0)=1/2$), and $E(h)$ is a small error term. Hunter and Regan \cite{HR72} (correcting the argument in \cite{CR68,MR71}) show, for $x >0$ with $x\neq \pi/h$, that
\begin{equation} \label{Ehbound}
|E(h)| \leq \frac{2|z \re^{-z^2}|\, \re^{-\pi^2/h^2}}{\pi^{1/2}(1-\re^{-2\pi^2/h^2})|x^2-\pi^2/h^2|}.
\end{equation}
Thus, for every fixed $z=x+\ri y$ with $x>0$, the modified trapezoidal rule approximation obtained by neglecting $E(h)$ in \eqref{I*(h,0)} is very rapidly convergent indeed as $h\to0$.

As the bound \eqref{Ehbound} suggests, neglecting the error term $E(h)$ in \eqref{I*(h,0)} gives a very accurate approximation also for $x=\mathrm{Re}(z)=0$, except that the approximation is undefined if $z=\ri kh$, for some $k\in \Z$, and there are stability issues in evaluation if $z$ is close to one of these points. It is suggested in \cite{HR72} to solve this issue by switching to the composite midpoint rule where needed. Precisely, Hunter and Regan \cite{HR72} propose to use the formula \eqref{I*(h,0)} (neglecting the error term $E(h)$), if $1/4\leq \varphi(y/h)\leq 3/4$, where $y=\mathrm{Im}(z)$ and
\begin{equation}\label{fract_fn}
\varphi(t):=t-\left\lfloor t\right\rfloor\in[0,1)
\end{equation}
is the fractional part of $t$. Otherwise they suggest to use the midpoint-rule-based formula
\begin{equation}\label{I*(h,1/2)}
\erfc(z) = \frac{hz\re^{-z^2}}{\pi}\sum_{k=-\infty}^{\infty}\frac{\re^{-(k+1/2)^2 h^2}}{z^2+(k+1/2)^2 h^2}+\frac{2\,H\!\left(\pi/h-x\right)}{1+\re^{2\pi z/h}} + E^\prime(h),
\end{equation}
neglecting the corresponding error term $E^\prime(h)$ which they show satisfies the same bound \eqref{Ehbound} as $E(h)$.

These proposals from \cite{HR72} are our starting point. In a practical implementation the sums in \eqref{I*(h,0)} and \eqref{I*(h,1/2)} must be truncated, say to $-N\leq k\leq N$. The contributions of this paper are to:
\begin{itemize}
\item[i)] convert a modified version of the proposals of Hunter and Regan \cite{HR72} into a fully-specified algorithm, making clear how the choice of $h>0$ should be related to $N$ for optimal accuracy;
\item[ii)] provide error estimates for the approximations we propose for $w(z)$, proving that the maximum absolute error (and the maximum relative error in the upper-half plane) decrease exponentially with $N$, reducing by a factor $\re^\pi\approx 23.1$ for each additional quadrature point;
\item[iii)] demonstrate that the claimed exponential convergence in absolute and relative errors is achieved numerically;
\item[iv)] present numerical results that suggest that the simple approximation formulae we propose are competitive in accuracy and computation time with existing methods for computing $w(z)$.
\end{itemize}

In carrying out i) we are addressing earlier criticisms of the algorithms in \cite{CR68, MR71, HR72} made by Weideman \cite{Weid94}, who observes that the formula \eqref{I*(h,0)} with the summation reduced to $-N\leq k\leq N$
\begin{displayquote}
\emph{``is very accurate, provided for given $z$ and $N$ the optimal step size $h$ is selected. It is not easy, however, to determine this optimal $h$ a priori.''}
\end{displayquote}
Our recommendations address this issue, detailing which approximation formula and what step size $h$ should be used for each $N$ and $z$.  (It turns out that, to obtain high accuracy, the approximation formula chosen must depend on both $z$ and $N$ but $h$ only needs to depend on $N$.)

The bounds we obtain in carrying out ii) prove that the absolute error in our approximation for $w(z)$ tends to zero exponentially with $N$, uniformly in the complex plane. This is a substantial improvement on the existing bound \eqref{Ehbound} which blows up when $x=\pi/h$, and does not capture the additional truncation errors due to replacing infinite by finite sums in the approximations \eqref{I*(h,0)} and \eqref{I*(h,1/2)}.

Concretely, our proposed approximation to $w(z)$, for $z=x+\ri y$, with $x,y\geq0$, is
\begin{eqnarray}\label{wNdef}
w_{N}(z):=
\begin{cases}
w_{N}^{\mathrm{M}}(z),&\quad \mbox{if } y \geq \max\left(x,\pi/h\right),\\
w_{N}^{\mathrm{MT}}(z),&\quad \mbox{if }y< x\quad\mbox{and}\quad 1/4\leq\varphi\left(x/h\right)\leq 3/4,\\
w_{N}^{\mathrm{MM}}(z),&\quad\mbox{otherwise},
\end{cases}
\end{eqnarray}
where $\varphi$ is defined by \eqref{fract_fn}, $N\in \N_0:= \N\cup\{0\}$,
\begin{eqnarray} \label{hdef}
h &:=&\sqrt{\pi\big/(N+1)}\,,\\ \label{IN1}
w_{N}^{\mathrm{M}}(z)&:=& \displaystyle\frac{2\ri\,h\,z}{\pi} \,\sum_{k=0}^{N} \frac{\re^{-t_{k}^{2}}}{ z^2-t_{k}^{2}},\\\label{IN2}
w_{N}^{\mathrm{MM}}(z)&:=&\displaystyle\frac{2\,\re^{-z^2}}{1+\re^{-2\ri\pi z/h}}+w_{N}^{\mathrm{M}}(z),\\\label{IN3}
w_{N}^{\mathrm{MT}}(z)&:=&\displaystyle\frac{2\,\re^{-z^2}}{1-\re^{-2\ri\pi z/h}}+\frac{\ri\,h}{\pi z}+\frac{2\ri\,h\,z}{\pi} \,\sum_{k=1}^{N}\frac{\re^{-\tau_{k}^{2}}}{ z^2-\tau_{k}^{2}},\\
\label{tkdef}
t_{k}&:=&(k+1/2)h,\quad \mbox{and} \quad \tau_{k}:=kh.
\end{eqnarray}

We extend the approximation to the full complex plane by using the symmetries \eqref{symm}, precisely by defining
\begin{equation} \label{symm2}
w_N(z):=\overline{w_N(-\overline{z})}, \mbox{ if } y \geq 0 \mbox{ and } x<0, \quad w_N(z):=2\re^{-z^2}-w_N(-z), \mbox{ if } y<0.
\end{equation}
We supply in Table SM1 of the supplementary materials \cite{ACW202} the Matlab code implementing the approximation $w_N(z)$ that we use for the computations in \S\ref{sec: w_num}\footnote{The codes in these supplementary materials are also available at github, see \url{https://github.com/sms03snc/Faddeeva}.}.

The main error estimate that we prove, using standard complex analysis arguments including a Phragm\'en-Lindel\"of principle, is
\begin{theorem}\label{thm: w_main}
Suppose $w_{N}(z)$ is given by \eqref{wNdef} and \eqref{symm2}. Then, for $N\in \N_0$ and $z=x+\ri y$,
{ \begin{eqnarray*}
|w(z)-w_{N}(z)|&\leq& C_1\,\re^{-\pi N}, \quad { \mbox{for all } x,y\in \R, \quad \mbox{and}}\\
\frac{|w(z)-w_{N}(z)|}{|w(z)|}&\leq& C_2\,\sqrt{N+1}\,\re^{-\pi N}, \quad { \mbox{if } x\in \R \mbox{ and } y\geq 0},
\end{eqnarray*}
where $C_1\approx 0.67$ is given by \eqref{C1def} and $C_2 \approx 3.97$ by \eqref{C2def}.}
\end{theorem}



We remark that approximation of $w(z)$, for $z\in \C$, provides an effective route to the computation of other special functions, including Fresnel integrals (e.g., \cite{Mohd14}), Dawson's integral \cite[Equation (7.5.1)]{DL10}, and the Voigt functions \cite[Equation (7.19.3)]{DL10}. Indeed, we have previously used, in the restricted case $\arg(z)=\pi/4$, an approximation resembling $w_N^{\mathrm{MM}}(z)$ when approximating Fresnel integrals \cite{Mohd14}, proving results in the spirit of Theorem \ref{thm: w_main}.

Let us summarise the rest of the paper. In the largest section, \S\ref{w_approx}, we derive the above formulae and error bounds. In \S\ref{sec:survey} we review the existing, alternative approximate methods for computing $\erfc(z)$ and $w(z)$ for complex $z$. For none of these has an error bound been proved, similar to Theorem \ref{thm: w_main}. In \S\ref{sec: w_num} we carry out numerical experiments that  confirm the accuracy of $w_N(z)$, showing that its absolute error is $<2\times 10^{-15}$ throughout the complex plane with $N=11$, and that the same bound holds for the relative error in the upper half-plane. We also show that our new approximation is competitive in accuracy and  appears to be competitive in computing times with  the methods that we survey in \S\ref{sec:survey}, specifically those of \cite{Weid94,Zag12,Zaghloul16,Abrar18}.

We note that this paper is based, in significant part, on Chapter 3 of the first author's thesis \cite{Mohd17}.

\section{The proposed approximation and its error bounds}\label{w_approx}
In this section we derive the approximation  given by \eqref{wNdef} based on modified trapezoidal rules. We also derive the error bounds of Theorem \ref{thm: w_main} that demonstrate that the absolute and relative errors in $w_N(z)$ both decrease exponentially as $N$ increases.

\subsection{The contour integral argument and its history} \label{sec:BasicHistory} Given any $f\in C(\R)$ that decays sufficiently rapidly at infinity, let
$$
I[f] := \int_{-\infty}^\infty f(t)\, d t,
$$
and, for $h>0$ and $\alpha\in[0,1)$, define the {\em generalised trapezoidal rule approximation} to $I[f]$ by
\begin{equation}\label{IhdefnC1}
I_{h,\alpha}[f]:=\displaystyle h\,\sum_{k\in\mathrm{Z}}f((k-\alpha)h).
\end{equation}
We note that $I_{h,\alpha}[f] = I_{h,0}[f_\alpha]$, where $f_\alpha(t):= f(t-\alpha h)$ for $t\in \R$, and that $I_{h,0}[f]$ is the trapezoidal rule approximation to $I[f]$ and $I_{h,1/2}[f]$ its composite midpoint rule approximation.

The approximation \eqref{IhdefnC1} for $I[f]$ converges exponentially when the integrand is analytic in a strip surrounding the real axis and has sufficient decay at $\pm\infty$. The derivation of this result, using contour integration and Cauchy's residue theorem, dates back, for a particular case, at least to Turing \cite{Tur45}, and has been analysed in more general cases by Goodwin \cite{Good49}, 
Davis \cite{Da:59}, McNamee \cite{Mc64}, Schwartz \cite{Sch69} and Stenger \cite{St73}. For a detailed history and discussion see Trefethen and Weideman \cite{TW14}.

The rate of exponential convergence depends on the width of the strip of analyticity around the real axis, and the accuracy of $I_{h,\alpha}[f]$ deteriorates when $f$ has singularities close to the real line. But, in the case when these singularities are poles, the contour integral method for establishing the exponential convergence of  $I_{h,\alpha}[f]$, that we will recall in Proposition \ref{prop:basic} below, leads naturally to corrections for modifying the trapezoidal rule and recovering rapid convergence, these corrections expressed in terms of residues of $f$ at these poles. This appears to have been observed explicitly first by Chiarella and Reichel \cite{CR68}, in the context of evaluating \eqref{erfc_int} (and see Matta and Reichel \cite{MR71}, Hunter and Regan \cite{HR72}, and Mori \cite{Mo83}), and has been developed into a general theory by Bialecki \cite{Bia89} (and see La Porte \cite{Scott07}).

It is convenient to recall in a proposition the standard arguments (\cite{CR68,MR71,HR72} and cf.\ \cite[pp.\ 402--403]{TW14}) that are made to prove exponential convergence, since we will use these arguments below. These use, for given $h>0$ and $\alpha\in[0,1)$, the function $g(\zeta)$ defined by
\begin{equation}\label{gdefnC1}
g(\zeta):=\ri\,\cot\left(\pi\left(\alpha + \zeta/h\right)\right).
\end{equation}
This is a meromorphic function with simple poles at $\zeta=(k-\alpha)h$, $k\in\Z$, and the properties that, for $\zeta\in \C$ with $\eta=\mathrm{Im}(\zeta)$,
\begin{equation}\label{gbound1C1}
|1-g(\zeta)|\leq \frac{2e^{-2\pi \eta/h}}{1- e^{-2\pi \eta/h}}, \;\; \mbox{ if } \eta>0, \quad |1+g(\zeta)|\leq \frac{2e^{2\pi \eta/h}}{1- e^{2\pi \eta/h}}, \;\; \mbox{ if } \eta<0.
\end{equation}
We also use, for $H\in \R$, the notation $\Gamma_H$ for the path $\{t+\ri H:t\in\R\}$ traversed in the direction of increasing $t$.  It is enough for our purposes to suppose that the poles of $f$ are simple. For the case of poles of arbitrary order see \cite[Theorem 2.2]{Bia89}.
\begin{prop}\label{prop:basic} Suppose that, for some $H>0$, $f$ is analytic in the strip $S_H:=\{\zeta\in \C:|\mathrm{Im}(\zeta)|<H\}$, except for a finite number of simple poles at $\zeta_1,\ldots,\zeta_m\in S_H$, with $\eta_k:=\mathrm{Im}(\zeta_k)\neq 0$, for $k=1,\ldots,m$. Suppose also that  $f$ is continuous in $\overline{S_H}\setminus\{\zeta_1,\ldots,\zeta_m\}$ and that, for some $r>1$, $f(\zeta)=O(|\zeta|^{-r})$ as $|\mathrm{Re}(\zeta)|\to \infty$, uniformly in $S_H$. Then
$$
I[f]-I_{h,\alpha}[f]= \frac{1}{2}\left(\int_{\Gamma_H}f(\zeta)(1-g(\zeta))\, d\zeta +  \int_{\Gamma_{-H}}f(\zeta)(1+g(\zeta))\, d\zeta\right) + C_{h,\alpha,H}[f],
$$
where
$$
C_{h,\alpha,H}[f] := \pi\ri \sum_{k=1}^m \left(\mathrm{sign}(\eta_k)-g(\zeta_k)\right)R_k,
$$
and $R_k := \mathrm{Res}(f,\zeta_k)=\lim_{\zeta\to \zeta_k}(\zeta-\zeta_k)f(\zeta)$ denotes the residue of $f$ at $\zeta_k$.
\end{prop}
\begin{figure}
\centering
\begin{tikzpicture}[scale=1]
\tikzset{myptr/.style={decoration={markings,mark=at position 1 with %
    {\arrow[scale=2]{>}}},postaction={decorate}}}
\draw [myptr] (-5,0)--(5.2,0);
\draw(5.2,-0.1) node[below]{$\mathrm{Re}$};
\draw[myptr](0,-2.5)--(0,3);
\draw(-0.1,3) node[left]{$\mathrm{Im}$};
\draw [thick] (-5,2)--(5,2);
\draw (0,2.2) node[left]{$H$};
\draw(4,2)node[above]{$\Gamma_H$};
\draw [thick] (-5,-2)--(5,-2);
\draw (0,-2.2) node[left]{$-H$};
\draw(4,-2)node[below]{$\Gamma_{-H}$};
\foreach \x in {-4.7,-4.2,...,5}
\filldraw (\x,0) circle (1pt);
\filldraw (-3,-1.5) circle (1pt);
\draw(-3,-1.5) node[below]{$\zeta_1$};
\filldraw (-1.8,-0.5) circle (1pt);
\draw(-1.8,-0.5) node[below]{$\zeta_2$};
\filldraw (-0.6,0.7) circle (1pt);
\draw(-0.6,0.7) node[below]{$\zeta_3$};
\draw (0.5,0.75) node{$\cdots$};
\filldraw (1.6,0.8) circle (1pt);
\draw(1.6,0.8) node[below]{$\zeta_{m-1}$};
\filldraw (2.8,1.1) circle (1pt);
\draw(2.8,1.1) node[below]{$\zeta_m$};
\draw[dashed,myptr] (-3.95,1.9)--(-3.95,-1.9)--(3.55,-1.9)--(3.55,1);
\draw(3.55,1.1)node[right]{$C_{\widetilde H,i,j}$};
\draw[dashed] (3.55,1)--(3.55,1.9)--(-3.95,1.9);
\end{tikzpicture}
\caption{The contour $C_{\widetilde H,i,j}$ used in the proof of Proposition \ref{prop:basic}. The dots on the real line are the poles of $g(\zeta)$ at $(k-\alpha)h$, for $k\in \Z$.} \label{fig:main}
\end{figure}
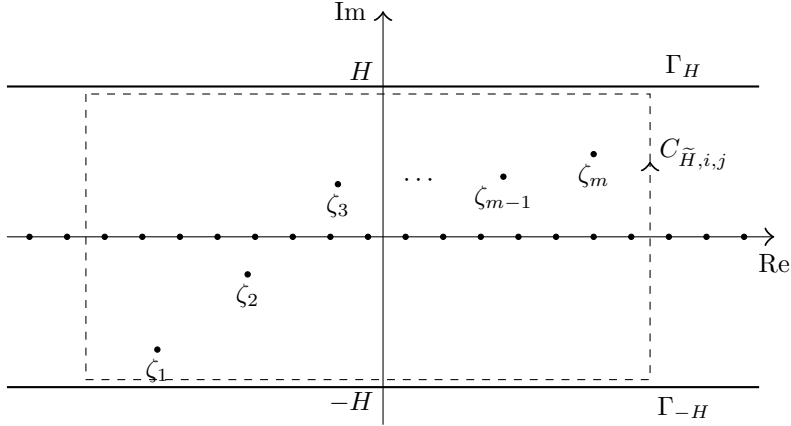
\begin{proof} Let $A_k := \left(k-\alpha + \frac{1}{2}\right)h$, for $k\in \Z$. Let $C_{\widetilde H,i,j}$ denote the positively oriented rectangular contour with corners at $A_i\pm \ri \widetilde H$ and $A_j\pm \ri \widetilde H$, choosing $\widetilde H\in (0,H)$ and the integers $i<0$ and $j>0$ so that $C_{\widetilde H,i,j}$ encloses the poles $\zeta_1,\ldots,\zeta_m$ (see Figure \ref{fig:main}). Noting that $\mathrm{Res}(g,(k-\alpha)h)= \ri h/\pi$, for $k\in \Z$, we apply Cauchy's residue theorem to
$$
\int_{C_{\widetilde H,i,j}}f(\zeta)g(\zeta) \, d\zeta.
$$
We take the limit $\widetilde H\to H$, and then the limit as $i\to -\infty$ and $j\to+\infty$ which leads (see \cite[pp.\ 402--403]{TW14} for more detail) to
$$
\int_{\Gamma_{-H}}f(\zeta)g(\zeta)\, d\zeta - \int_{\Gamma_H}f(\zeta)g(\zeta)\, d\zeta  =-2I_{h,\alpha}[f]+ 2\pi\ri \sum_{k=1}^m g(\zeta_k)R_k.
$$
Making similar applications of Cauchy's residue theorem we obtain also that
$$
\int_{\Gamma_H}f(\zeta)\, d\zeta +  \int_{\Gamma_{-H}}f(\zeta)\, d\zeta = 2I[f] - 2\pi\ri \sum_{k=1}^m \mathrm{sign}(\eta_k)R_k,
$$
and the result follows.
\end{proof}

Thanks to the bounds \eqref{gbound1C1} it follows from the above proposition that
\begin{equation}\label{eq:mainbound}
|I[f]-I^*_{h,\alpha,H}[f]|\leq  \frac{e^{-2\pi H/h}}{1- e^{-2\pi H/h}}\int_{\Gamma_H}\left(|f(\zeta)|+|f(-\zeta)|\right)\, |d\zeta|,
\end{equation}
where
\begin{equation}\label{eq:ModDef}
I^*_{h,\alpha,H}[f]:= I_{h,\alpha}[f] + C_{h,\alpha,H}[f]
\end{equation}
is what we will call the {\em modified generalised trapezoidal rule}. If $f$ is analytic in $S_H$, so that $I^*_{h,\alpha,H}[f]= I_{h,\alpha}[f]$, this bound reduces to \cite[Theorem 5.1]{TW14}, and proves that $I_{h,\alpha}[f]$ is exponentially convergent as $h\to 0$. In the more general case that $f$ has simple pole singularities in $S_H$, the bound \eqref{eq:mainbound} proves the same rate of exponential convergence for $I^*_{h,\alpha,H}[f]$, the trapezoidal rule modified to take into account these poles of $f$.

Our application of the above proposition and bound will be to the integrals given by \eqref{erfc_int} and \eqref{wint}.
In these cases (cf.\ Goodwin \cite{Good49}) we have additionally that
\begin{equation} \label{eq:fF}
f(\zeta)= \re^{-\zeta^2}F(\zeta),
\end{equation}
where $F$ is even and $F(\zeta)=O(1)$ as $|\mathrm{Re}(\zeta)|\to\infty$, uniformly in $S_H$. This satisfies the conditions of the above proposition, and, since $|\exp(-\zeta^2)|=\re^{H^2-t^2}$ for $\zeta=t+\ri H$, and $\int_{-\infty}^\infty \exp(-t^2)\, dt = \sqrt{\pi}$, the bound \eqref{eq:mainbound} implies in this case that
\begin{equation}\label{eq:mainbound2}
\left|I[f]-I^*_{h,\alpha,H}[f]\right|\leq  \frac{2\re^{H^2-2\pi H/h}}{1- \re^{-2\pi H/h}}\, \int_{-\infty}^\infty \re^{-t^2} |F(t+\ri H)|\, dt \leq \frac{2\sqrt{\pi}M\re^{H^2-2\pi H/h}}{1- \re^{-2\pi H/h}},
\end{equation}
where
\begin{equation} \label{eq:Mdef}
M:= \sup_{t\in \R}|F(t+\ri H)|.
\end{equation}

An important observation, particularly when applying the above bound in cases where $F$ is meromorphic in the whole complex plane, is that the exponent $H^2-2\pi H/h$ is minimised by the choice $H=\pi/h$, in which case  $H^2-2\pi H/h = -\pi^2/h^2$. The representations \eqref{I*(h,0)} and \eqref{I*(h,1/2)}, with the bound \eqref{Ehbound} (which bounds both $|E(h)|$ and $|E^\prime(h)|$), follow from applying \eqref{eq:mainbound2} to \eqref{erfc_int}. Precisely, we apply \eqref{eq:mainbound2} with $H=\pi/h$ and $\alpha =0$ and $\alpha=1/2$, in the respective cases \eqref{I*(h,0)} and \eqref{I*(h,1/2)}, noting that \eqref{erfc_int} can be written as $\erfc(z)=I[f]$ with $f$ given by \eqref{eq:fF} and $F(t) := z\exp(-z^2)/(\pi(z^2+t^2))$. To obtain the bound \eqref{Ehbound} from \eqref{eq:mainbound2} and \eqref{eq:Mdef} we note further that, for $z=x+\ri y$ and $\zeta = t+\ri \pi/h$, $|z^2+\zeta^2|=|z+\ri\zeta|\,|z-\ri\zeta|\geq |x-\pi/h| |x+\pi/h| =|x^2-\pi^2/h^2|$. (We remark that the blow-up when $x=\pi/h$ in the bound \eqref{Ehbound} is an artefact of the method of argument, that the contour $\Gamma_H$ with $H=\pi/h$ passes through a pole of the function $F$ when $x=\pi/h$.)

\subsection{Trapezoidal rule error estimates}
In this subsection we apply the above methods and bounds to obtain uniform absolute and relative error estimates as a replacement for the bound \eqref{Ehbound}. Except where explicitly indicated otherwise, all the results we prove hold for all $h>0$ and all $\alpha\in [0,1)$. In the estimates we obtain in Proposition \ref{prop:new1} we avoid the equivalent of the blow-up seen in \eqref{Ehbound} when $x=\pi/h$  by making two separate applications of the contour integral argument, with $H=\pi/h\pm \epsilon$. It turns out that the $\epsilon$ correction to the optimal choice $H=\pi/h$ incurs only an $O(1)$ factor correction to the error bound as long as $\epsilon$ is fixed independently of $h$. At the same time, if we take  a fixed $\epsilon$ and switch as needed between $H=\pi/h+\epsilon$ and $H=\pi/h-\epsilon$ we can ensure that $\Gamma_H$ remains at least an $O(1)$ distance from the poles of the integrand as we vary $z$.

For $z=x+\ri y$ with $y>0$ we
write \eqref{wint} as $w(z)=I[f_z]$ where
\begin{eqnarray}\label{w_f}
f_z(t):= e^{-t^2}F_z(t)\quad\mbox{and}\quad F_z(t):=\frac{\ri\,z}{\pi(z^2-t^2)}.
\end{eqnarray}
The even function $f_z$ is meromorphic with simple poles at $t = \pm z$ and residues
\begin{eqnarray}\label{w_res}
R_{1}:=\mathrm{Res}\left(f_z,z \right)= \frac{-\ri\,\re^{-z^2}}{2\pi}\quad\mbox{and}\quad R_{2}:=\mathrm{Res}\left(f_z,-z \right)=-R_{1}.
\end{eqnarray}
Thus, using the notation of Proposition \ref{prop:basic},
\begin{equation}\label{CF1.1}
C_{h,\alpha,H}[f_z] =\left\{\begin{array}{ll}
                          2\re^{-z^2}/(1-\re^{-2\ri\pi(\alpha + z/h)}), & \mbox{if } H>y, \\
                          0, & \mbox{if } H<y,
                        \end{array}\right.
\end{equation}
and the trapezoidal rule approximation to $w(z)=I[f_z]$ is
\begin{eqnarray}\label{w_Ih}
\;\;\;\;\;\; I_{h,\alpha}[f_z]=h\sum_{k\in\mathbb{Z}}\frac{\ri z\,\re^{-(k-\alpha)^2h^2}}{\pi(z^2-(k-\alpha)^2h^2)} = \left\{\begin{array}{ll}
                          {\displaystyle\frac{h\ri z}{\pi}\left[\frac{1}{z^2}+2\sum_{k=1}^\infty\frac{\re^{-\tau_k^2}}{z^2-\tau_k^2}\right]}, & \alpha=0, \\
                          {\displaystyle\frac{2h\ri z}{\pi}\sum_{k=0}^\infty\frac{\re^{-t_k^2}}{z^2-t_k^2}}, & \alpha=\frac{1}{2},
                        \end{array}\right.
\end{eqnarray}
where $t_k$ and $\tau_k$ are as defined in \eqref{tkdef}. It is useful also to introduce the notations
\begin{equation} \label{eq:Cf2}
C_{h,\alpha}[f_z] :=\lim_{H\to\infty} C_{h,\alpha,H}[f_z]=
                          \frac{2\re^{-z^2}}{1-\re^{-2\ri\pi(\alpha + z/h)}},
\end{equation}
and
\begin{equation} \label{eq:If*2}
I^*_{h,\alpha}[f_z] :=\lim_{H\to\infty} I^*_{h,\alpha,H}[f_z]= I_{h,\alpha}[f_z] + C_{h,\alpha}[f_z],
\end{equation}
where $I^*_{h,\alpha,H}[f_z]$ is the modified trapezoidal rule defined by \eqref{eq:ModDef}. Note that
\begin{equation} \label{eq:same}
C_{h,\alpha}[f_z] = C_{h,\alpha,H}[f_z] \mbox{ and } I^*_{h,\alpha}[f_z] = I^*_{h,\alpha,H}[f_z], \mbox{ for } y<H,
\end{equation}
and that $I^*_{h,\alpha,H}[f_z]=I_{h,\alpha}[f_z]$ for $y>H$. We note also that, as a function of $z$, $I^*_{h,\alpha}[f_z]$ is entire: $I_{h,\alpha}[f_z]$ and $C_{h,\alpha}[f_z]$ are both meromorphic with simple poles at $z=\pm (k-\alpha)h$, $k\in \Z$, and it is easy to check that the pole contributions cancel in the sum \eqref{eq:If*2}: the singularities in  $I^*_{h,\alpha}[f_z]$ are removable.

The following proposition (cf.~\cite[\S2]{Mo83}) bounds $|w(z)-I^{*}_{h,\alpha,H}[f_z]|$ for $H=\pi/h$. It also bounds the relative error
$|w(z)-I^{*}_{h,\alpha,H}[f_z]|/|w(z)|$ using the lower bound \cite[Theorem 6]{Mohd14}
\begin{equation}\label{wlb}
|w(z)|\geq\frac{1}{1+\sqrt{\pi}|z|},\quad\mathrm{Im}(z)\geq0,
\end{equation}
this being sharp for small and large $z$ as $w(0)=1$ and $w(z)\sim \ri/(\sqrt{\pi}z)$ as $z\to\infty$ \cite[(2.6)]{WG70}.

\begin{proposition} \label{prop:new1} Suppose that $z=x+\ri y$ with $0\leq x\leq y$. Then
\begin{equation} \label{eq:fnb5}
\left|w(z)-I^{*}_{h,\alpha}[f_z]\right| \leq 2\sqrt{\frac{\re}{\pi}}\,\frac{\re^{-\pi^2/h^2}}{1- \re^{-2\pi^2/h^2}}
\end{equation}
and
\begin{equation} \label{eq:fnb5a}
\frac{\left|w(z)-I^{*}_{h,\alpha}[f_z]\right|}{|w(z)|} \leq \frac{4\sqrt{2\pi\re}}{h}\,\frac{\re^{-\pi^2/h^2}}{1- \re^{-2\pi^2/h^2}},
\end{equation}
if  $0\leq y\leq \pi/h$, while
\begin{equation} \label{eq:fnb6}
\left|w(z)-I_{h,\alpha}[f_z]\right| \leq 4\sqrt{\frac{\re}{\pi}}\,\frac{\re^{-\pi^2/h^2}}{1- \re^{-2\pi^2/h^2+\sqrt{2}\pi/h}}
\end{equation}
and
\begin{equation} \label{eq:fnb6a}
\frac{\left|w(z)-I_{h,\alpha}[f_z]\right|}{|w(z)|} \leq \frac{4\sqrt{2\pi\re}(1+\sqrt{\pi})}{h}\,\frac{\re^{-\pi^2/h^2}}{1- \re^{-2\pi^2/h^2+\sqrt{2}\pi/h}},
\end{equation}
if $y\geq \pi/h$ and $h<\sqrt{2}\, \pi$.
\end{proposition}
\begin{proof} For $H>0$ and $y>0$ with $H\neq y$ we have the bound \eqref{eq:mainbound2} with
\begin{equation} \label{eq:Mform}
M= \sup_{t\in \R}|F_z(t+\ri H)| = \frac{|z|}{\pi}\,\sup_{t\in \R} \frac{1}{|z^2-(t+\ri H)^2|}.
\end{equation}
Since  $|z^2-(t+\ri H)^2|=|z-t-\ri H|\, |z+t+\ri H|\geq |y-H|\, |y+H|$, it follows that
$
M \leq |z|/(\pi|y^2-H^2|)
$,
so that
\begin{equation} \label{eq:fnb1}
\left|w(z)-I^{*}_{h,\alpha,H}[f_z]\right| \leq \frac{2|z|\,\re^{H^2-2\pi H/h}}{\sqrt{\pi}\left(1- \re^{-2\pi H/h}\right)|y^2-H^2|}\leq \frac{2\sqrt{2}\,y\,\re^{H^2-2\pi H/h}}{\sqrt{\pi}\left(1- \re^{-2\pi H/h}\right)|y^2-H^2|},
\end{equation}
since $0\leq x\leq y$ so that $|z|\leq \sqrt{2}y$. Further,
 using \eqref{wlb} and $0\leq x\leq y$,
\begin{equation} \label{eq:fnb2}
\frac{\left|w(z)-I^{*}_{h,\alpha,H}[f_z]\right|}{|w(z)|} \leq \frac{2\sqrt{2}(1+\sqrt{2\pi}y)y\,\re^{H^2-2\pi H/h}}{\sqrt{\pi}\left(1- \re^{-2\pi H/h}\right)|y^2-H^2|}.
\end{equation}

Now suppose that $0<y\leq \pi/h$ and take $H=\pi/h+\varepsilon$ for some $\varepsilon>0$. Then, by \eqref{eq:same}, $I^*_{h,\alpha}[f_z] = I^*_{h,\alpha,H}[f_z]$, and since $y/(H^2-y^2)$ and $(1+\sqrt{2\pi} y)y/(H^2-y^2)$ are increasing as functions of $y$ on $[0,H)$, it follows from \eqref{eq:fnb1} and \eqref{eq:fnb2} with $H=\pi/h+\varepsilon$ that
\begin{equation} \label{eq:fnb3}
\left|w(z)-I^{*}_{h,\alpha}[f_z]\right| \leq \frac{2\sqrt{2\pi}\,\re^{-\pi^2/h^2+\varepsilon^2}}{\left(1- \re^{-2\pi^2/h^2-2\pi\varepsilon/h}\right)\varepsilon(2\pi+\varepsilon h)}
\end{equation}
and
\begin{equation} \label{eq:fnb4}
\frac{\left|w(z)-I^{*}_{h,\alpha}[f_z]\right|}{|w(z)|} \leq \frac{2\sqrt{2\pi}\,(h+\sqrt{2\pi}\pi)\,\re^{-\pi^2/h^2+\varepsilon^2}}{h\left(1- \re^{-2\pi^2/h^2-2\pi\varepsilon/h}\right)\varepsilon(2\pi+\varepsilon h)}.
\end{equation}
Choosing $\varepsilon = 1/\sqrt{2}$ to minimise $\exp(\varepsilon^2)/\varepsilon$ we obtain, for $0<y<\pi/h$, the bound \eqref{eq:fnb5}, and also the bound \eqref{eq:fnb5a} on noting that $(h+\sqrt{2\pi}\, \pi)/(2\pi+h/\sqrt{2})\leq \sqrt{2}$; these bounds hold also for $y=0$ and $y=\pi/h$ since the left hand sides of the bounds depend continuously on $y$ on $[0,\pi/h]$ (recall that $I^{*}_{h,\alpha}[f_z]$ is an entire function of $z$ and that $w(z)$ is bounded below on $y\geq 0$ by \eqref{wlb}).

Now suppose that $y>\pi/h$ and take $H=\pi/h-\varepsilon$ for some $\varepsilon\in (0,\pi/h)$. Then $I^*_{h,\alpha,H}[f_z]=I_{h,\alpha}[f_z]$, and since $y/(y^2-H^2)$ and $(1+\sqrt{2\pi} y)y/(y^2-H^2)$ are decreasing as functions of $y$ on $(H,\infty]$, it follows from \eqref{eq:fnb1} and \eqref{eq:fnb2} with $H=\pi/h-\varepsilon$ that
\begin{equation} \label{eq:fnb7}
\left|w(z)-I_{h,\alpha}[f_z]\right| \leq \frac{2\sqrt{2\pi}\,\re^{-\pi^2/h^2+\varepsilon^2}}{\left(1- \re^{-2\pi^2/h^2+2\pi\varepsilon/h}\right)\varepsilon(2\pi-\varepsilon h)}
\end{equation}
and
\begin{equation} \label{eq:fnb8}
\frac{\left|w(z)-I_{h,\alpha}[f_z]\right|}{|w(z)|} \leq \frac{2\sqrt{2\pi}\,(h+\sqrt{2\pi}\pi)\,\re^{-\pi^2/h^2+\varepsilon^2}}{h\left(1- \re^{-2\pi^2/h^2+2\pi\varepsilon/h}\right)\varepsilon(2\pi-\varepsilon h)}.
\end{equation}
If $\pi/h>1/\sqrt{2}$ we can again choose $\varepsilon = 1/\sqrt{2}$, obtaining the bounds \eqref{eq:fnb6} and \eqref{eq:fnb6a} for $y>\pi/h$; these bounds hold also for $y=\pi/h$ since the left hand sides of the bounds depend continuously on $y$ on $[\pi/h,\infty)$.
\end{proof}

It follows immediately from the definition \eqref{eq:Cf2} that, for $x\in \R$, $y>0$,
\begin{eqnarray}\label{w_Ch}
\left|C_{h,\alpha}[f_z]\right|\leq\frac{2\,\re^{-2\pi y/h}}{1- \re^{-2\pi y/h}}\,\re^{y^2-x^2}.
\end{eqnarray}
Since $|I^*_{h,\alpha}[f_z]|\leq |I_{h,\alpha}[f_z]| + |C_{h,\alpha}[f_z]|$, the following corollary follows from the above proposition, \eqref{w_Ch}, and \eqref{wlb}. The purpose of this corollary is to provide bounds on errors on part of the boundary of the infinite sector $\Omega$ (as defined in Lemma \ref{MMP} with $a=4$). This will lead to absolute and relative error bounds for $I^*_{h,\alpha}[f_z]$ as an approximation on the whole of $\overline \Omega$ via the Phragm\'en--Lindel\"of principle of Lemma \ref{MMP} -- see Proposition \ref{prop:new3} below.

\begin{corollary} \label{cor:xybound}
If $z=x+\ri y$ with $x=y\geq 0$ and $h<\sqrt{2}\, \pi$, then
\begin{equation} \label{eq:fnb10}
\left|w(z)-I^*_{h,\alpha}[f_z]\right| \leq c_a\, \frac{\re^{-\pi^2/h^2}}{1- \re^{-2\pi^2/h^2+\sqrt{2}\pi/h}}
\end{equation}
and
\begin{equation} \label{eq:fnb11}
 \frac{\left|w(z)-I^*_{h,\alpha}[f_z]\right|}{|w(z)|} \leq \frac{c_r}{h}\,\frac{\re^{-\pi^2/h^2}}{1- \re^{-2\pi^2/h^2+\sqrt{2}\pi/h}},
\end{equation}
where
\begin{equation} \label{eq:cardef}
c_a := \frac{2(2\re+\sqrt{\pi})}{\sqrt{\re\pi}}\approx 4.934 \quad \mbox{and} \quad c_r:= \frac{2\sqrt{2\pi}\,(1+\sqrt{\pi})(2\re+\sqrt{\pi})}{\sqrt{\re}}\approx 60.77.
\end{equation}
\end{corollary}
\begin{proof} For $0\leq x=y\leq \pi/h$ these bounds follow immediately from the sharper bounds \eqref{eq:fnb5} and \eqref{eq:fnb5a}.  Suppose now that $x=y\geq \pi/h$ and $h<\sqrt{2}\, \pi$. Then  it follows from \eqref{w_Ch} that
$$
\left|C_{h,\alpha}[f_z]\right|\leq\frac{2\,\re^{-2\pi^2/h^2}}{1- \re^{-2\pi^2/h^2}} \leq \frac{2\,\re^{-\pi^2/h^2}}{\sqrt{\re}\left(1- \re^{-2\pi^2/h^2}\right)}.
$$
Further, since $(1+\sqrt{2\pi} y)/(\re^{2\pi y/h}-1)$ is decreasing as a function of $y$ on $[\pi/h,\infty)$, it follows from \eqref{w_Ch} and \eqref{wlb} that
$$
\frac{\left|C_{h,\alpha}[f_z]\right|}{|w(z)|}\leq\frac{2\,(h+\sqrt{2\pi} \pi)\re^{-2\pi^2/h^2}}{h\left(1- \re^{-2\pi^2/h^2}\right)} \leq \frac{2\sqrt{2}\, \pi(1+\sqrt{\pi})\,\re^{-\pi^2/h^2}}{h\sqrt{\re}\left(1- \re^{-2\pi^2/h^2}\right)}.
$$
Since $|I^*_{h,\alpha}[f_z]|\leq |I_{h,\alpha}[f_z]| + |C_{h,\alpha}[f_z]|$, the required bounds for  $x=y\geq \pi/h$ follow from the above bounds and  \eqref{eq:fnb6} and \eqref{eq:fnb6a} in Proposition \ref{prop:new1}.
\end{proof}

Proposition \ref{prop:new1} tells us that $I^*_{h,\alpha,H}[f_z]$, with $H=\pi/h$, is an approximation to $w(z)$ with   controllable absolute and relative errors for $\pi/4\leq \mathrm{arg}(z)\leq \pi/2$. To complement this result we will show in Proposition \ref{prop:new3} below that $I^*_{h,\alpha}[f_z]$ is an approximation to $w(z)$ with   controllable absolute and relative errors for $0\leq \mathrm{arg}(z)\leq \pi/4$, so that $I^*_{h,\alpha,\pi/h}[f_z]$ and $I^*_{h,\alpha}[f_z]$ together provide trapezoidal rule-based approximations to $w(z)$ across the whole first-quadrant (and, via the symmetries \eqref{symm}, across the whole complex plane). We will prove Proposition \ref{prop:new3} via the following Phragm\'en--Lindel\"of principle applied with $a=4$ to the left hand sides of \eqref{eq:fnb7} and \eqref{eq:fnb8}.
\begin{lemma}{\cite[Chapter VI, Cor.~4.2]{Con78}}\label{MMP}
Let $a\geq 1/2$ and put
$$
\Omega:=\left\{z=r\re^{\ri\theta}:r>0 \mbox{ and } 0<\theta <\pi/a\right\}.
$$
Suppose that $f$ is analytic on $\Omega$ and continuous in $\overline{\Omega}$ and that there is a constant $P$ such that $|f(z)|\leq P$ for all $z\in\partial \Omega$. If there are positive constants $Q$ and $b<a$ such that
$\left|f(z)\right|\leq Q\,\exp(|z|^b)
$ for all $z\in \Omega$, then $|f(z)|\leq P$ for all $z\in\overline{\Omega}$.
\end{lemma}
The main step in proving Proposition \ref{prop:new3} via this lemma is to show that the left hand sides of \eqref{eq:fnb10} and \eqref{eq:fnb11} are bounded on $\partial \Omega$ when $a=4$. We have bounded these left hand sides already on $\{r\re^{\ri\pi/4}:r\geq 0\}$ in Corollary \ref{cor:xybound}. It remains to bound them on the positive real axis which we do in the next proposition.

\begin{proposition} \label{prop:new2}
If $x\geq 0$ then
\begin{equation} \label{eq:fnbx1}
\left|w(x)-I^*_{h,\alpha}[f_x]\right| \leq \frac{2h\re^{-\pi^2/h^2}}{\pi^{3/2}\left(1- \re^{-2\pi^2/h^2}\right)}
\end{equation}
and
\begin{equation} \label{eq:fnbx2}
 \frac{\left|w(x)-I^*_{h,\alpha}[f_x]\right|}{|w(x)|} \leq \left[8+\frac{10h}{\pi^{3/2}}\right] \,  \frac{\re^{-\pi^2/h^2}}{1- \re^{-2\pi^2/h^2}}.
\end{equation}
\end{proposition}
\begin{proof}
Arguing as in the proof of Proposition \ref{prop:new1}, for $x\geq 0$ and $0<y<\pi/h$ we have the bound \eqref{eq:mainbound2} with $H=\pi/h$. Thus, and since (recall \eqref{eq:same}) $I^*_{h,\alpha,H}[f_z]=I^*_{h,\alpha}[f_z]$ for $y<H$, it follows that
\begin{equation} \label{eq:xfirst}
\left|I[f]-I^*_{h,\alpha}[f_z]\right|\leq  \frac{2\re^{-\pi^2/h^2}}{1- \re^{-2\pi^2/h^2}}\, \int_{-\infty}^\infty \re^{-t^2} |F_z(t+\ri H)|\, dt \leq \frac{2\sqrt{\pi}M\re^{-\pi^2/h^2}}{1- \re^{-2\pi^2/h^2}},
\end{equation}
where $M$ is given by \eqref{eq:Mform}. Since $M$ and the left and right hand sides of the first of these inequalities depend continuously on $y$ on $[0,\pi/h)$, the above inequalities hold also for $y=0$. Since also, for $x\geq 0$ and $t\in \R$ we have that
\begin{eqnarray} \nonumber
|x^2-(t+\ri\pi/h)^2| &= &|x^2-(|t|+\ri\pi/h)^2|\\ \label{eq:xtbound}
&=&|x-|t|-\ri \pi/h|\,|x+|t|+\ri \pi/h|\geq\frac{\pi}{h}|x+\ri\pi/h| \geq \frac{\pi}{h}|x|,
\end{eqnarray}
it follows that $M\leq h/\pi^2$ for $x\geq 0$, and \eqref{eq:fnbx1} follows from \eqref{eq:xfirst} with $z=x\geq0$.

From \eqref{wlb} and \eqref{eq:xfirst}, with $z=x\geq 0$ and $H=\pi/h$, it follows that
\begin{equation} \label{eq:xsecond}
\frac{\left|I[f]-I^*_{h,\alpha}[f_x]\right|}{|w(x)|}\leq  \frac{2(1+\sqrt{\pi}\, x)\re^{-\pi^2/h^2}}{1- \re^{-2\pi^2/h^2}}\, \int_{-\infty}^\infty G_x(t)\, dt
\end{equation}
for $x\geq 0$, where
$$
G_x(t) := \re^{-t^2} |F_x(t+\ri \pi/h)| = \frac{x\re^{-t^2}}{\pi\,|x^2-(t+\ri\pi/h)^2|}.
$$
If $x\geq 0$ and $-x/2\leq t\leq x/2$, then
$$
|x^2-(t+\ri\pi/h)^2| = |x-t-\ri \pi/h|\,|x+t+\ri \pi/h| \geq |x/2-\ri \pi/h|\,|x/2+\ri \pi/h|= \frac{x^2+4\pi^2/h^2}{4},
$$
so that
$$
\int_{-x/2}^{x/2}G_x(t)\, dt  \leq \frac{4x}{\pi(x^2+4\pi^2/h^2)}\int_{-\infty}^\infty \re^{-t^2}\, dt =  \frac{4x}{\sqrt{\pi}(x^2+4\pi^2/h^2)}.
$$
But also, using \eqref{eq:xtbound}, we see that
$
G_x(t) \leq h\re^{-t^2}/\pi^2$,
for all $x\geq 0$, $t\in\R$, so that
$$
\int_{\R\setminus[-x/2,x/2]} G_x(t)\, dt \leq \frac{2h}{\pi^2}\int_{x/2}^\infty \re^{-t^2}\, dt  <\frac{2h}{\pi^2x}\re^{-x^2/4} \leq \frac{2h}{\pi^2x},
$$
for $x>0$ since, for $a>0$,
\begin{equation}\label{Integ_prop}
\int_a^\infty \re^{-t^2}\, dt = \frac{\re^{-a^2}}{2a}-\frac{1}{2}\int_a^\infty \frac{\re^{-t^2}}{t^2}d t< \frac{\re^{-a^2}}{2a}.
\end{equation}
Moreover,
$$
\int_{\R\setminus[-x/2,x/2]} G_x(t)\, dt \leq \frac{2h}{\pi^2}\int_{0}^\infty \re^{-t^2}\, dt  = \frac{h}{\pi^{3/2}}.
$$
Thus, and since $\min(a,b)\leq 2ab/(a+b)$ if $a\geq 0$, $b\geq 0$, and $a+b>0$, it follows that
$$
\int_{\R\setminus[-x/2,x/2]} G_x(t)\, dt \leq \frac{h}{\pi^{2}}\, \min(\sqrt{\pi},2x^{-1}) \leq \frac{4h}{\pi^{3/2}(2+\sqrt{\pi}\,x)}
$$
so that
$$
\int_{-\infty}^{\infty}G_x(t)\, dt  \leq  \frac{4}{\sqrt{\pi}}\left(\frac{x}{x^2+4\pi^2/h^2}+\frac{h}{\pi(2+\sqrt{\pi}\,x)}\right)
$$
and \eqref{eq:fnbx2} follows from \eqref{eq:xsecond}, on noting that  $x^2/(x^2+4\pi^2/h^2)\leq 1$ and $x/(x^2+4\pi^2/h^2)\leq h/(4\pi)$.
\end{proof}

In our final proposition of this subsection  (cf.~\cite[Proposition 3.3.4]{Mohd17}) we combine Corollary \ref{cor:xybound}, Proposition \ref{prop:new2}, and Lemma \ref{MMP} to bound approximations to $w(z)$ in $0\leq \arg(z)\leq \pi/4$, complementing the bounds in Proposition \ref{prop:new1} for $\pi/4\leq \arg(z)\leq \pi/2$.

\begin{proposition} \label{prop:new3}
Suppose that $h<\sqrt{2}\, \pi$ and $z=x+\ri y$ with $0\leq y\leq x$. Then the bounds \eqref{eq:fnb10} and \eqref{eq:fnb11} hold with $c_a$ and $c_r$ given by \eqref{eq:cardef}.
\end{proposition}
\begin{proof} We will prove this proposition by applying Lemma \ref{MMP} with $a=4$, so that $\Omega=\{z=x+\ri y:0< y <x\}$, to the functions $E_h(z):= w(z)-I^*_{h,\alpha}[f_z]$ and $e_h(z) := E_h(z)/w(z)= (w(z)-I^*_{h,\alpha}[f_z])/w(z)$.

We have remarked already that $w(z)$ and  $I^*_{h,\alpha}[f_z]$ are entire as a function of $z$, so that $E_h$ is entire and, noting \eqref{wlb}, $e_h$ is analytic in $\mathrm{Im}(z)>0$ and continuous in $\mathrm{Im}(z)\geq 0$. In particular, $E_h$ and $e_h$ are continuous in $\overline{\Omega}$ and analytic in $\Omega$. Further, if $h<\sqrt{2}\, \pi$, it follows from Corollary \ref{cor:xybound} and Proposition \ref{prop:new2}, on noting that the bounds in Proposition \ref{prop:new2} are smaller than those in Corollary \ref{cor:xybound}, that $E_h$ and $e_h$ satisfy the bounds claimed in the proposition when $z\in\partial \Omega$, i.e.\ for $z$ on $\{x+\ri x:x\geq 0\}$ and on the positive real axis. Thus the proposition follows by Lemma \ref{MMP} if we can show that $E_h(z)$ and $e_h(z)$ do not grow too rapidly as $z\to \infty$ in $\Omega$.

But, if $h<\sqrt{2} \, \pi$, it follows from \eqref{eq:same} and \eqref{eq:fnb1} applied with $H=3$ that, for some constant $C>0$ independent of $z$, $|E_h(z)|\leq C |z|$ if $z\in \Omega$ with $y\leq 2$. Similarly, since $I^*_{h,\alpha,H}[f_z]=I_{h,\alpha}[f_z]$ if $y>H$ and $I^*_{h,\alpha}[f_z] = I_{h,\alpha}[f_z]+C_{h,\alpha}[f_z]$, it follows from \eqref{eq:fnb1} applied with $H=1$ and \eqref{w_Ch} that, for some constant $\widetilde C>0$ independent of $z$, $|E_h(z)|\leq \widetilde C |z|$ if $z\in \Omega$ with $y\geq 2$. Thus $|E_h(z)|\leq C^*|z|$ for $z\in \Omega$, where $C^*:= \max(C,\widetilde C)$, so that also, applying \eqref{wlb}, $|e_h(z)|\leq C^*|z|(1+\sqrt{\pi}\, |z|)$ for $z\in \Omega$. Thus the proposition follows by applying Lemma \ref{MMP}.
\end{proof}

  The following corollary summarises and simplifies, at the cost of a little sharpness, the results of Propositions \ref{prop:new1} and \ref{prop:new3} and of this subsection.
\begin{corollary} \label{cor:int} Suppose that $z=x+\ri y$ with $x\geq 0$, $y\geq 0$, and $h< \sqrt{2}\, \pi$. Then the bounds \eqref{eq:fnb10} and \eqref{eq:fnb11} hold with $c_a$ and $c_r$ given by \eqref{eq:cardef} if $y\leq \max(x,\pi/h)$. The same bounds hold as bounds on $|w(z)-I_{h,\alpha}[f_z]|$ and $|w(z)-I_{h,\alpha}[f_z]|/|w(z)|$, respectively, with the same values of $c_a$ and $c_r$, if $y\geq \max(x,\pi/h)$.
\end{corollary}
\begin{proof}
The first claim of the corollary follows from Proposition \ref{prop:new3} and \eqref{eq:fnb5} and \eqref{eq:fnb5a}, and the second follows from \eqref{eq:fnb6} and \eqref{eq:fnb6a}.
\end{proof}

\subsection{Truncating the infinite series} Propositions \ref{prop:new1} and \ref{prop:new3} together provide accurate trapezoidal-rule-based approximations to $w(z)$ in the first quadrant of the complex plane, that can be extended to the whole complex plane using the symmetries \eqref{symm}. But in implementation the infinite series in these approximations must be truncated. We estimate the additional error this introduces in this subsection.

At this point, since we wish to use the fact that $f_z$ is even to reduce computation, we restrict attention to the cases $\alpha=0$ and $\alpha=1/2$, in which cases the trapezoidal rule approximation reduces to \eqref{w_Ih}. In these cases we approximate $I_{h,\alpha}[f_z]$ by
\begin{eqnarray}\label{w_IN}
I^{N}_{h,\alpha}[f_z]:=
\begin{cases}
hf_z(0)+\displaystyle 2h\sum_{k=1}^{N}f_z(\tau_{k}), & \alpha=0\\
\displaystyle 2h\sum_{k=0}^{N}f_z(t_{k}), & \alpha=1/2,
\end{cases}
\end{eqnarray}
with $f_z$ given by \eqref{w_f} and $\tau_{k}$ and $t_{k}$ defined in \eqref{tkdef}.
We will call the error in approximating $I_{h,\alpha}[f_z]$ by $I^{N}_{h,\alpha}[f_z]$ the truncation error, given by
\begin{equation}\label{w_TN}
T^{N}_{h,\alpha}[f_z]:=\displaystyle 2h\,\sum_{k=N+1}^{\infty}f_z(s_k),
\end{equation}
where $s_k:= (k+\alpha)h$.
This is also the error in approximating $I^*_{h,\alpha}[f_z]$ by $I^{*,N}_{h,\alpha}[f_z]$, where $I^*_{h,\alpha}[f_z]$ is defined in \eqref{eq:If*2} and
\begin{equation} \label{eq:I*Ndef}
I^{*,N}_{h,\alpha}[f_z]:=I^{N}_{h,\alpha}[f_z]+C_{h,\alpha}[f_z].
\end{equation}

The following result (\cite[Proposition 3.3.7]{Mohd17}) bounds $T^{N}_{h,\alpha}[f_z]$ for $\pi/4\leq \arg(z)\leq \pi/2$. We use in this proposition the estimate, obtained since $\exp(-t^2)$ is decreasing on $(0,\infty)$ and noting \eqref{Integ_prop}, that
\begin{equation} \label{eq:sumbound}
2h\sum_{k=M}^{\infty} \re^{-s_k^2} \leq  2h \re^{-s_M^2} + 2\int_{s_M}^\infty \re^{-t^2}\, dt \leq \frac{2hs_M+1}{s_M} \,\re^{-s_M^2}, \quad M\in \N.
\end{equation}
\begin{proposition}\label{pro:w_TN3}
Suppose $\alpha=0$ or $1/2$ and $z=x+\ri y$ with $y\geq x\geq0$. Then, for $N\in \N_0$,
\begin{eqnarray} \label{eq:TN1}
|T^{N}_{h,\alpha}[f_z]|&\leq& \frac{(1+2h\,\tau_{N+1})}{\pi \tau_{N+1}^2}\, \re^{-\tau_{N+1}^2}\quad\mbox{and}\\ \label{eq:TN2}
\frac{|T^{N}_{h,\alpha}[f_z]|}{|w(z)|}&\leq&\frac{(1+2h\,\tau_{N+1})(1+2\sqrt{\pi}\,\tau_{N+1})}{\pi \tau_{N+1}^2}\,\re^{-\tau_{N+1}^2}.
\end{eqnarray}
\end{proposition}
\begin{proof}
For $z=x+\ri y$ with $y\geq x \geq0$,
\begin{eqnarray*}
|z^2 -s_{k}^2|^{2}=y^4+s_{k}^4+x^4+2x^2y^2+2s_{k}^2(y^2-x^2)\geq y^4+s_k^4\geq y^4+\tau_k^4.
\end{eqnarray*}
Thus, and recalling \eqref{w_f} and using  \eqref{eq:sumbound} with $M=N+1$,
\begin{eqnarray*}
|T^{N}_{h,\alpha}[f_z]|& \leq & \frac{2\sqrt{2}h\,y}{\pi} \sum_{k=N+1}^\infty \frac{\re^{-\tau_k^2}}{\sqrt{y^4+\tau_k^4}}\leq  \frac{\sqrt{2}y\,(1+2h\,\tau_{N+1})}{\pi\tau_{N+1}\sqrt{y^4+\tau_{N+1}^4}}\,\re^{-\tau_{N+1}^2}.
\end{eqnarray*}
Moreover,
\begin{eqnarray*}
y\Big/\sqrt{y^4+\tau_{N+1}^4}\leq\frac{1}{\sqrt{2}\,\tau_{N+1}}\quad\mbox{and}\quad y^2\Big/\sqrt{y^4+\tau_{N+1}^4}\leq 1,
\end{eqnarray*}
so that \eqref{eq:TN1} follows and also
\begin{eqnarray} \label{eq:TN3}
y\,|T^{N}_{h,\alpha}[f_z]|\leq \frac{\sqrt{2}(1+2h\,\tau_{N+1})}{\pi \tau_{N+1}}\,\re^{-\tau_{N+1}^2}.
\end{eqnarray}
Since, by \eqref{wlb}, $|w(z)|^{-1}\leq 1+\sqrt{2\pi} y$ for $0\leq x\leq y$, \eqref{eq:TN2} follows from \eqref{eq:TN1} and \eqref{eq:TN3}.
\end{proof}

The following result (\cite[Propositions 3.3.5, 3.3.6]{Mohd17}) bounds $T^{N}_{h,\alpha}[f_z]$ for $0\leq \arg(z)$ $\leq \pi/4$, so that Propositions \ref{pro:w_TN3} and \ref{pro:w_TN1} together bound the absolute and relative truncation errors in the first quadrant. The case $0\leq \arg(z)\leq \pi/4$ is more subtle because $T^{N}_{h,\alpha}[f_z]$ is unbounded, it has simple poles at $z=s_k$, for $k\geq N+1$, and our bound requires  that the distance of $z$ from this set of poles is $\geq h/4$. Despite this restriction, we can construct accurate approximations covering the whole region $0\leq \arg(z)\leq \pi/4$ because $s_k=\tau_k=kh$ for $\alpha=0$ while $s_k=t_k=\tau_k + h/2$ for $\alpha=1/2$, so that $\{z:|z-s_k|\geq h/4 \mbox{ for either } s_k=\tau_k \mbox{ or } t_k, \mbox{ for }k\geq N+1\}$ includes the whole of $\{z:0\leq \arg(z)\leq \pi/4\}$.

\begin{proposition}\label{pro:w_TN1}
Suppose $\alpha=0$ or $1/2$ and $z=x+\ri y$ with $0\leq y\leq x$ and $|z-s_{k}|\geq h/4$ for $k\geq N+1$. Then, for $N\in \N_0$,
\begin{eqnarray}\label{TN1}
|T^{N}_{h,\alpha}[f_z]|&\leq& c(h,N,\alpha) := \frac{2\sqrt{2}\,(1+2hs_{N+1})\,(h+4s_{N+1})}{\pi h\,s_{N+1}^2}\,\re^{-s_{N+1}^2}\quad\mbox{and}\\ \label{TN1a}
\displaystyle\frac{|T^{N}_{h,\alpha}[f_z]|}{|w(z)|}&\leq&(1+\sqrt{2\pi}s_{N+1})\,c(h,N,\alpha).
\end{eqnarray}
\end{proposition}
\begin{proof} From \eqref{w_f} and \eqref{w_TN}, for $0\leq y\leq x$,
\begin{equation} \label{eq:TNb}
|T^{N}_{h,\alpha}[f_z]|\leq \frac{2h|z|}{\pi} \sum_{k=N+1}^\infty \frac{\re^{-s_k^2}}{|z^2-s_k^2|} \leq \frac{2\sqrt{2}hx}{\pi(x+s_{N+1})} \sum_{k=N+1}^\infty \frac{\re^{-s_k^2}}{|z-s_k|}.
\end{equation}
Thus, and noting \eqref{wlb}, the bounds \eqref{TN1} and \eqref{TN1a} hold if $x=0$.

Choose $\theta$ with $0<\theta<1$. Recalling $z=x+\ri y$, in the case that $x>0$
let $M$ be the smallest integer $\geq N+1$ such that $s_M> \theta x$, so that, if $M>N+1$, it holds that $s_k\leq \theta x$ and $|z-s_k|\geq (1-\theta)x$ for $k<M$. If  $M>N+1$ it follows, using the bound \eqref{eq:sumbound}, that
$$
2hx\sum_{k=N+1}^{M-1} \frac{\re^{-s_k^2}}{|z-s_k|} \leq \frac{2h}{1-\theta}\, \sum_{k=N+1}^{\infty}\re^{-s_k^2} \leq \frac{2hs_{N+1}+1}{(1-\theta)s_{N+1}} \,\re^{-s_{N+1}^2},
$$
while, for $M\geq N+1$, assuming $|z-s_{k}|\geq h/4$ for $k\geq N+1$ and again using \eqref{eq:sumbound},
$$
2hx\sum_{k=M}^{\infty} \frac{\re^{-s_k^2}}{|z-s_k|} \leq  8x\sum_{k=M}^{\infty}\re^{-s_k^2} \leq \frac{4x(2hs_{M}+1)}{hs_M} \re^{-s_{M}^2}\leq \frac{4(2hs_{M}+1)}{h\theta } \,\re^{-s_{M}^2}.
$$
Thus, and since $(2ht+1)\exp(-t^2)$ is decreasing as a function of $t$ on $[h,\infty)$ and $s_M\geq s_{N+1} \geq h$,
\begin{eqnarray} \nonumber
2hx\sum_{k=N+1}^{\infty} \frac{\re^{-s_k^2}}{|z-s_k|} &\leq &(2hs_{N+1}+1)\,\left[\frac{4}{h\theta }  + \frac{1}{(1-\theta)s_{N+1}} \right]\,\re^{-s_{N+1}^2}\\ \label{eq:finalsb}
 &=&\frac{2(2hs_{N+1}+1)(h+4s_{N+1})}{hs_{N+1}}\,\re^{-s_{N+1}^2},
\end{eqnarray}
on choosing $\theta = 4s_{N+1}/(h+4s_{N+1})$ so as to equalise the terms in the square brackets. From \eqref{eq:TNb} and \eqref{eq:finalsb}, on noting that $x+s_{N+1}\geq s_{N+1}$ and $x/(x+s_{N+1})\leq 1$, we see that, for $x>0$, \eqref{TN1} holds and also $x|T^{N}_{h,\alpha}[f_z]| \leq s_{N+1}c(h,N)$. From these bounds and that $|w(z)|^{-1}\leq 1+\sqrt{2\pi} x$ for $0\leq y\leq x$ by \eqref{wlb}, the bound \eqref{TN1a} follows.
\end{proof}

The following corollary summarises and simplifies, at the cost of a little sharpness, the results of Propositions \ref{pro:w_TN3} and \ref{pro:w_TN1} and this subsection.
\begin{corollary} \label{cor:trunc} Suppose $\alpha=0$ or $1/2$, $z=x+\ri y$ with $x\geq 0$, $y\geq 0$, and either $y\geq x$ or $|z-s_k|\geq h/4$ for $k\geq N+1$. Then, for $N\in \N_0$,
\begin{eqnarray}\label{TN1f}
|T^{N}_{h,\alpha}[f_z]|&\leq& c(h,N,0) = \frac{2\sqrt{2}\,(1+2h\tau_{N+1})\,(h+4\tau_{N+1})}{\pi h\,\tau_{N+1}^2}\,\re^{-\tau_{N+1}^2},\quad\mbox{and also}\\ \label{TN1af}
\displaystyle\frac{|T^{N}_{h,\alpha}[f_z]|}{|w(z)|}&\leq&(1+\sqrt{2\pi}\,\tau_{N+1})\,c(h,N,0), \quad \mbox{provided} \quad h \geq 1/(N+1).
\end{eqnarray}
\end{corollary}
\begin{proof} Proposition \ref{pro:w_TN3} implies the above bounds hold when $y\geq x$. If  $y\leq x$ the above bounds are immediate from Proposition \ref{pro:w_TN1} in the case $\alpha=0$. They hold also when $\alpha=1/2$ as: i) $c(h,N,1/2)\leq c(h,N,0)$, since $s^{-m}\re^{-s^2}$ decreases as $s$ increases on $(0,\infty)$ for $m=0,1,2$; ii) $t_{N+1}c(h,N,1/2)\leq \tau_{N+1}c(h,N,0)$ if $\tau_{N+1}\geq 1$ (i.e., $h\geq 1/(N+1)$), since also $s\re^{-s^2}$ decreases as $s$ increases on $[1,\infty)$.
\end{proof}

\subsection{Proof of the main theorem} In this subsection we bring together the bounds on the error in the trapezoidal rule approximation (Corollary \ref{cor:int}) and on the truncation error (Corollary \ref{cor:trunc}) to bound the errors in the truncated trapezoidal rule approximation \eqref{w_IN} and its modification \eqref{eq:I*Ndef}. Clearly,
\begin{eqnarray} \label{tri1}
|w(z)- I_{h,\alpha}^N[f_z]| &\leq &|w(z) - I_{h,\alpha}[f_z]| + |T_{h,\alpha}^N[f_z]| \quad \mbox{and}\\ \label{tri2}
|w(z)- I_{h,\alpha}^{*,N}[f_z]| &\leq &|w(z) - I^*_{h,\alpha}[f_z]| + |T_{h,\alpha}^N[f_z]|.
\end{eqnarray}
Applying Corollaries \ref{cor:int} and \ref{cor:trunc} it follows that, for $\alpha = 0$ and $1/2$ and $z=x+\ri y$ with $x\geq 0$, $0\leq y\leq \max(x,\pi/h)$,
\begin{eqnarray} \label{eq:finalb1}
|w(z)- I_{h,\alpha}^{*,N}[f_z]| &\leq & c_a\, \frac{\re^{-\pi^2/h^2}}{1- \re^{-2\pi^2/h^2+\sqrt{2}\pi/h}} + c(h,N,0), \quad \mbox{where}\\ \nonumber
c(h,N,0) &:=& \frac{2\sqrt{2}\,(1+2h\tau_{N+1})\,(h+4\tau_{N+1})}{\pi h\,\tau_{N+1}^2}\,\re^{-\tau_{N+1}^2},
\end{eqnarray}
provided $y\geq x$ or $|z-s_k| \geq h/4$ for $k\geq N+1$. Similarly,  applying Corollaries \ref{cor:int} and \ref{cor:trunc},
\begin{eqnarray} \label{eq:finalb2}
|w(z)- I_{h,\alpha}^{N}[f_z]| &\leq & c_a\, \frac{\re^{-\pi^2/h^2}}{1- \re^{-2\pi^2/h^2+\sqrt{2}\pi/h}} + c(h,N,0)
\end{eqnarray}
for $\alpha = 0$ and $1/2$, if $z=x+\ri y$ with $x\geq 0$ and $y\geq \max(x,\pi/h)$.

We choose the  step size $h$, as a function of $N$, to approximately balance the contributions from the trapezoidal rule error and the truncation error in the above error bounds. Precisely, we choose $h$ so that the exponents of $\re^{-\pi^2/h^2}$ and $e^{-\tau_{N+1}^2}$ are equal, i.e.\ we define $h:=\sqrt{\pi/(N+1)}$ as in \eqref{hdef}. With this choice of $h$ we have $\pi/h=\tau_{N+1} = \sqrt{(N+1)\pi}$,
\begin{equation} \label{eq:cwithhchoice}
c(h,N,0) =\frac{2\sqrt{2}\,(1+2\pi)\,(5+4N)}{\pi^2(N+1)}\,\re^{-(N+1)\pi} \leq \frac{10\sqrt{2}\,(1+2\pi)}{\pi^2}\,\re^{-(N+1)\pi},
\end{equation}
and
\begin{eqnarray} \label{eq:etermb}
\hspace*{3ex} \frac{\re^{-\pi^2/h^2}}{1- \re^{-2\pi^2/h^2+\sqrt{2}\pi/h}} = \frac{\exp(-(N+1)\pi)}{1- \exp\left(-2(N+1)\pi+\sqrt{2(N+1)\pi}\,\right)} \leq c^* \, \re^{-(N+1)\pi},
\end{eqnarray}
where
\begin{equation} \label{eq:c*def}
c^* := \left(1- \exp\left(-2\pi+\sqrt{2\pi}\right)\right)^{-1} \approx 1.0234.
\end{equation}
This leads to our main theorem.
\begin{theorem} \label{thm:main1}
Suppose $\alpha=0$ or $1/2$, $N\in \N_0$, $h$ is defined by \eqref{hdef}, and $z=x+\ri y$. If $\pi/h \geq y\geq x\geq 0$ or $0\leq y\leq x$ and $|z-s_k|\geq h/4$ for $k\geq N+1$, then
\begin{eqnarray}\label{final1}
|w(z)- I_{h,\alpha}^{*,N}[f_z]|\leq C_1\, \re^{-\pi N}\;\;\;\mbox{and} \;\;\; 
\frac{|w(z)- I_{h,\alpha}^{*,N}[f_z]|}{|w(z)|}\leq  C_2\, \sqrt{N+1}\,\re^{-\pi N},
\end{eqnarray}
where
\begin{eqnarray} \label{C1def}
C_1 &:=& \frac{2(2\re+\sqrt{\pi})}{\re^\pi\sqrt{\re\pi}\left(1- \exp\left(-2\pi+\sqrt{2\pi}\right)\right)}+ \frac{10\sqrt{2}\,(1+2\pi)}{\re^\pi\pi^2} \approx 0.6692 \quad \mbox{and}\\ \label{C2def}
C_2 & := & \frac{2\sqrt{2}\,(1+\sqrt{\pi})(2\re+\sqrt{\pi})}{\re^\pi\sqrt{\re}\left(1- \exp\left(-2\pi+\sqrt{2\pi}\right)\right)}+ \frac{10(1+2\pi)(2\pi+\sqrt{2})}{\re^\pi\pi^2} \approx 3.971.
\end{eqnarray}
Further, if $x\geq 0$ and $y\geq \max(x,\pi/h)$, then
\begin{eqnarray}\label{final3}
|w(z)- I_{h,\alpha}^{N}[f_z]|\leq C_1\, \re^{-\pi N}\;\;\; \mbox{and} \;\;\;
\frac{|w(z)- I_{h,\alpha}^{N}[f_z]|}{|w(z)|}\leq  C_2\, \sqrt{N+1}\,\re^{-\pi N}.
\end{eqnarray}
\end{theorem}
\begin{proof} The bounds on $|w(z)- I_{h,\alpha}^{*,N}[f_z]|$ and $|w(z)- I_{h,\alpha}^{N}[f_z]|$ follow from \eqref{eq:finalb1}, \eqref{eq:finalb2}, \eqref{eq:cwithhchoice}, \eqref{eq:etermb}, and the definitions, \eqref{eq:cardef} and \eqref{eq:c*def}, of $c_a$ and $c^*$.  With $h$ defined by \eqref{hdef}, it follows that $h\geq 1/(N+1)$ and
$$
1+\sqrt{2\pi}\,\tau_{N+1} = 1+ \pi\sqrt{2(N+1)} \leq \frac{\sqrt{2}+2\pi}{\sqrt{2}}\, \sqrt{N+1}.
$$
Thus Corollaries \ref{cor:int} and \ref{cor:trunc}, together with \eqref{tri1}, \eqref{tri2}, \eqref{eq:cwithhchoice}, and \eqref{eq:etermb}, imply, with $h$ defined by \eqref{hdef}, that
$|w(z)- I_{h,\alpha}^{*,N}[f_z]|/|w(z)|$ and $|w(z)- I_{h,\alpha}^{N}[f_z]|/|w(z)|$ are both bounded above by
$$
\left(\frac{c_rc^*}{h} + \frac{10(1+2\pi)(2\pi + \sqrt{2})}{\pi^2}\, \sqrt{N+1}\right)\,\re^{-(N+1)\pi} = C_2 \, \sqrt{N+1}\,\re^{-\pi N},
$$
under their respective constraints on $x$ and $y$.
\end{proof}

The above theorem justifies approximating $w(z)$ by $I_{h,\alpha}^{N}[f_z]$, with $\alpha=0$ or $1/2$, if $h$ is given by \eqref{hdef}, $x\geq 0$ and $y\geq \max(x,\pi/h)$; we choose, arbitrarily, the midpoint-rule-based approximation $w_N^{\mathrm{M}}(z) := I_{h,1/2}^{N}[f_z]$, given explicitly by \eqref{IN1}. If $y< \max(x,\pi/h)$, the above theorem suggests approximating by $I_{h,\alpha}^{*,N}[f_z]$, with $\alpha=0$ or $1/2$. For $0\leq x\leq y<\pi/h$ we choose the modified midpoint-rule-based approximation $w_N^{\mathrm{MM}}(z):= I_{h,1/2}^{*,N}[f_z]$, given explicitly by \eqref{IN2}. This choice ensures that the distance of $z$ from the set of quadrature points $\{t_0,\ldots,t_N\}$ is $\geq h/(2\sqrt{2})$, so that the size of the largest term in the sum \eqref{IN1} does not exceed $1/(\pi\sqrt{2})$ and there is no loss of precision through cancellation between the two terms in the sum \eqref{IN2}. For $0\leq y<x$ we approximate either by $w_N^{\mathrm{MM}}(z)$ or by the modified trapezoidal-based approximation $w_N^{\mathrm{MT}}(z):= I_{h,0}^{*,N}[f_z]$, written explicitly in \eqref{IN3}. Which of these we use is determined by the rule \eqref{wNdef}. This rule ensures that $|z-s_k|\geq h/4$ for $k\in \N_0$, where $s_k=t_k$ when $\alpha=1/2$, $s_k=\tau_k$ when $\alpha=0$, so that Theorem \ref{thm:main1} applies and, in the use of both \eqref{IN2} and \eqref{IN3}, we avoid loss of precision through cancellation between nearly equal terms.

\begin{proof}[Proof of Theorem  \ref{thm: w_main}]
Noting the discussion in the above paragraph, the bounds in Theorem \ref{thm: w_main} follow, for $x\geq 0$, $y\geq 0$, immediately from those in Theorem \ref{thm:main1}. That the absolute error bound holds in the whole complex plane, and that the bound on the relative error holds in $\{z:y\geq 0\}$, follows from the bounds in the first quadrant and the symmetry relations and definitions \eqref{symm} and \eqref{symm2}.
\end{proof}

\section{Survey of existing methods} \label{sec:survey} There are a number of other schemes for computation of $w(z)$ for complex $z$ and we briefly summarise the best of
these, making connections with \eqref{wNdef}. Most use variations on polynomial or rational approximation, with different schemes in different regions of the first quadrant (leading, through \eqref{symm}, to approximation in the whole complex plane). Indeed, our own approximation \eqref{wNdef} uses three formulae \eqref{IN1}--\eqref{IN3}, with \eqref{IN1} rational and \eqref{IN2}--\eqref{IN3} rational with meromorphic corrections in terms of exponential functions.

Gautschi \cite{WG70} advocated, for larger $z$, the rational approximation
\begin{equation} \label{eq:nthconvergent}
w(z) \approx \frac{\ri/\sqrt{\pi}}{z-}\frac{1/2}{z-}\frac{2/2}{z-}\frac{3/2}{z-}\cdots \frac{(n-1)/2}{z},
\end{equation}
 the $n$th convergent of the beautiful Laplace continued fraction representation for $w(z)$ (specifically
suggesting $n=9$). Gautschi notes that: i) by construction the $n$th convergent is asymptotically accurate, with error $O(|z|^{-2n-1})$ as $|z|\to \infty$, uniformly in the first quadrant; ii) the $n$th convergent converges to $w(z)$ as $n\to\infty$ if and only if $\mathrm{Im}(z)>0$; iii) remarkably, for $\mathrm{Im}(z)>0$, the $n$th convergent coincides with the approximation obtained by approximating \eqref{wint} by an $n$-point Gauss-Hermite rule. For smaller $z$ Gautschi \cite{WG70} proposed (rational) approximations that are truncated Taylor expansions with the coefficients approximated by convergents of continued fractions.

This methodology, carefully tuned, is the basis of TOMS Algorithm 680 (Poppe and Wijers \cite{PW90})
 which achieves a relative error of $10^{-14}$ over nearly all the complex plane using, in the first quadrant: i) Maclaurin polynomials of degree $\leq 55$ for the odd function $\erf(-\ri z)$ (substituted into \eqref{wdef}) in an ellipse around the origin; ii) the convergents \eqref{eq:nthconvergent} with $n\leq 18$ outside a larger ellipse; iii) the more expensive mix of Taylor expansion and continued fraction calculation proposed by Gautschi \cite{WG70} in between. This algorithm has been used as a benchmark by several later authors.

Weideman \cite{Weid94,Weid95} proposed (the derivation starts from \eqref{wint}) the single rational approximation
\begin{equation}\label{weid}
w(z)\approx\frac{1}{\sqrt{\pi}(L-\ri z)}+\frac{2}{(L-\ri z)^2}\sum_{n=0}^{N-1} a_{n+1}\left(\frac{L+\ri z}{L-\ri z}\right)^n, \quad \mbox{for } \mathrm{Im}(z)\geq 0,
\end{equation}
where the size of $N$ controls the accuracy of the approximation, $L:=2^{-1/4}N^{1/2}$, and the $a_n$ are discrete Fourier coefficients that can be precomputed by the FFT.
He argues, based on operation counts, that, for intermediate values of $|z|$, the work required to compute $w(z)$ to $10^{-14}$ relative accuracy is much smaller for \eqref{weid} than for the Poppe and Wijers algorithm \cite{PW90}.

Zaghloul and Ali proposed a method (see TOMS Algorithm 916 \cite{Zag12} and the refinements in \cite{Zaghloul16}, and cf.\ \cite{Salzer51} and \cite[(7.1.29)]{AS64}) starting from
\begin{equation}\label{erf_zaghloul}
\erf(z)=\erf(x) +\frac{2\re^{-x^2}}{\sqrt{\pi}}\int_{0}^{y}\re^{t^2}\,\sin(2xt)\,dt+\frac{2\ri\,\re^{-x^2}}{\sqrt{\pi}}\int_{0}^{y}\re^{t^2}\cos(2xt)\,dt,
\end{equation}
for $z=x+\ri y$. They approximate
\begin{eqnarray}\label{Zag_approx1}
w(z)\approx u(x,y)+\ri v(x,y), \quad x,y\geq0,
\end{eqnarray}
where
\begin{eqnarray*}\nonumber
u(x,y)&:=& \re^{-x^2}\mathrm{erfcx}(y)\cos(2xy)+\frac{2a\sin^2(xy)}{\pi y}\,\re^{-x^2}+\frac{ay}{\pi}\left(-2\cos(2xy)\,S_{1}+S_{2}+S_{3}\right),\\
v(x,y)&:=& -\re^{-x^2}\mathrm{erfcx}(y)\sin(2xy)+\frac{a\sin(2xy)}{\pi y}\,\re^{-x^2}+\frac{a}{\pi}\left(2y\sin(2xy)\,S_{1}-S_{4}+S_{5}\right),
\end{eqnarray*}
$\erfcx(y):=\re^{y^2}\erf(y)$, and $S_j$, $j=1,\ldots,5$, are the following summations reminiscent of the trapezoidal rule approximations \eqref{I*(h,0)}:
\begin{equation}\label{Zag_sums}
\begin{split}
S_{1}&:=\sum_{k=1}^{\infty}\left(\frac{1}{a^2k^2+y^2}\right)\,\re^{-(a^2k^2+x^2)}, \quad
S_{2}:=\sum_{k=1}^{\infty}\left(\frac{1}{a^2k^2+y^2}\right)\,\re^{-(ak+x)^2},\\
S_{3}&:=\sum_{k=1}^{\infty}\left(\frac{1}{a^2k^2+y^2}\right)\,\re^{-(ak-x)^2}, \quad
S_{4}:=\sum_{k=1}^{\infty}\left(\frac{ak}{a^2k^2+y^2}\right)\,\re^{-(ak+x)^2},\\
S_{5}&:=\sum_{k=1}^{\infty}\left(\frac{ak}{a^2k^2+y^2}\right)\,\re^{-(ak-x)^2}.
\end{split}
\end{equation}
The authors have supplied us with their Matlab implementation \verb+Faddeyeva_v2(z,M)+ \cite{Zaghloul16}, where the parameter $M$ is the number of accurate significant figures required, and the code enforces $4\leq M\leq13$. In this code the choice $a=1/2$ is made and the sums in \eqref{Zag_sums} are truncated, the number of terms retained depending on $M$. Zaghloul and Ali \cite{Zag12} (and see \cite{Zaghloul16}) present numerical evidence that the approximation \eqref{Zag_approx1}, with $a=1/2$ and appropriate truncation of the infinite sums \eqref{Zag_sums}, is more accurate and faster than TOMS Algorithm 680 \cite{PW90}.  This algorithm and code has been used by Zaghloul \cite{Zag17} to benchmark a more efficient, low accuracy ($<4\times 10^{-5}$ maximum relative error in both real and imaginary parts) approximation to $w(z)$ for $z$ in the first quadrant.

Abrarov et al.\ \cite{Abrar18} (and see \cite{Abrar15}) proposed recently another method for computing $w(z)$ using modified rational approximations, namely
\begin{eqnarray}\label{abrar_approx3}
w(z)\approx
\begin{cases}
\psi_{1}(z),\mbox{ if } z\in D_{1},\\
\psi_{2}(z),\mbox{ if }z\in D_{2},\\
\psi_{3}(z),\mbox{ if }z\in D_{3},
\end{cases}
\end{eqnarray}
where $D_{1}:=\{z=x+\ri y:|z|<8\mbox{ and } y> 0.05x\}$, $D_{2}:=\{z=x+\ri y:|z|<8\mbox{ and } y\leq 0.05x\}$, $D_{3}:=\{z:|z|\geq 8\}$,
\begin{eqnarray}\label{abrar_approx1}
\hspace*{4ex} \psi_{1}(z):=\sum_{m=1}^{M}\displaystyle\frac{A_{m}+ B_{m}(z+\ri\alpha/2)}{C^2_{m}-(z+\ri\alpha/2)^2}, \quad \psi_{2}(z):=\re^{-z^2}+z\sum_{m=1}^{M+2}\displaystyle\frac{\alpha_{m}- \beta_{m}z^2}{\gamma_{m}-\theta_{m}z^2+z^4},
\end{eqnarray}
the coefficients $A_{m}$, $B_{m}$, $C_{m}$, $\alpha_{m}$, $\beta_{m}$, $\gamma_{m}$ and $\theta_{m}$ are specified in \cite{Abrar18}, and $\psi_3(z)$ is approximately \eqref{eq:nthconvergent} with $n=10$ (see \cite[Equation (9)]{Abrar18}).
Abrarov et al.\ \cite{Abrar18} present numerical evidence to show that \eqref{abrar_approx3} achieves an accuracy of $10^{-13}$ using $\alpha=2.75$ and $M=23$ in \eqref{abrar_approx1}. 




\section{Numerical results}\label{sec: w_num}
In this section we show calculations that illustrate and support Theorem \ref{thm: w_main}, and that compare the accuracy and efficiency of our approximation $w_{N}(z)$ given by \eqref{wNdef} to those of the approximations \eqref{weid}, \eqref{Zag_approx1}, and \eqref{abrar_approx3}. We omit comparison with the method of \cite{Zag17} because of its limited accuracy, and omit comparison with the older algorithm of \cite{PW90} because of evidence, discussed in \S3, that the approximation \eqref{Zag_approx1} is superior.

\begin{figure}
\includegraphics[width=13cm]{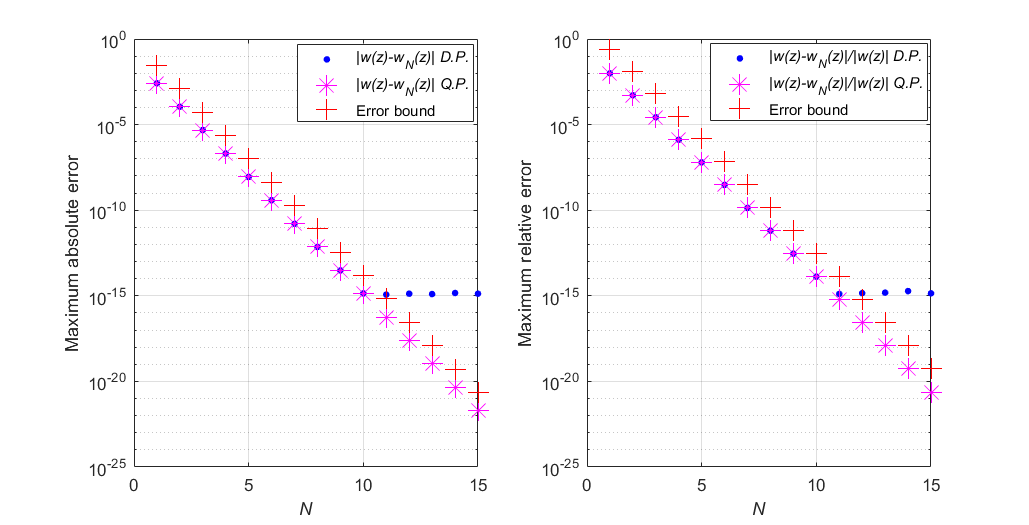}
\caption{Maximum absolute and relative errors in the approximation \eqref{wNdef} and the error bounds of Theorem \ref{thm: w_main}, plotted against $N$,  with $w_N(z)$ calculated in standard double precision (D.P.) and in quadruple precision (Q.P.). The maxima are taken over
a large number of points $z$ in the first quadrant, ranging from exponentially small to exponentially
large (see text for details).
}
\label{w_fig1}
\end{figure}

In Figure \ref{w_fig1} we plot estimates of the maximum values in the first quadrant of the absolute and relative errors in our approximation \eqref{wNdef} to $w(z)$. We show results for two implementations. The first of these uses the Matlab code \verb+wTrap(z,N)+, provided in Table SM1 of the supplementary material to this paper \cite{ACW202}, which computes our approximation $w_N(z)$ using standard double precision arithmetic in Matlab. The second implementation uses the Matlab code \verb+wTrap_Q(z,N)+ in Table SM4 of  \cite{ACW202} which computes $w_N(z)$ in quadruple precision arithmetic 
using the Multiprecision Computing Toolbox ADVANPIX (\url{www.advanpix.com}). In each case the maximum values we plot are discrete maxima taken over the $1,602,801$ points $z=10^{p}e^{i\theta}$, with $p=-6(0.0006)6$ and $\theta=0(\pi/1600)\pi/2$, a superset of the $40,401$ test values in Weideman \cite{Weid94,Weid95}. Whichever approximation for $w_{N}(z)$ is used, we use as the exact value for $w(z)$ the approximation $w_{20}(z)$, given by \eqref{wNdef} and computed by a call to \verb+wTrap_Q(z,N)+ with $N=20$. This approximation is predicted by Theorem \ref{thm: w_main} to have absolute and relative errors of $<3.5\times 10^{-28}$ and $<9.4\times 10^{-27}$, respectively, throughout the first quadrant. We expect these error bounds to be achieved working in quadruple precision arithmetic where the machine epsilon is $2^{-112}\approx 1.9\times 10^{-34}$.

We observe in Figure \ref{w_fig1} the rate of exponential convergence predicted by Theorem \ref{thm: w_main}. When evaluated in standard double precision arithmetic the approximation $w_{N}(z)$ achieves, with $N=11$ over this set of discrete points in the first quadrant, maximum absolute and relative errors which are $< 2\times 10^{-15}$.

\begin{table}[h]
\centering
\begin{tabular}{ |c||c|c|r|  }
 \hline
 Algorithm & \makecell{Maximum\\ abs.\ error} & \makecell{Maximum\\ rel.\ error} & \makecell{Computing\\ time (seconds)}\\
 \hline
 \verb+wTrap(z,11)+     & $1.19\times10^{-15}$  & $1.31\times10^{-15}$& $4.29\,\, (\pm 0.08)$  \\
 \hline
 \verb+cef(z,40)+ \cite{Weid94} & $1.33\times10^{-15}$   & $1.33\times10^{-15}$&$4.20\,\, (\pm 0.02)$  \\
 \hline
\verb+fadsamp(z)+  \cite{Abrar18} & $3.78\times10^{-14}$   & $3.78\times10^{-14}$&$5.74\,\,(\pm 0.04)$  \\
 \hline
 \verb+Faddeyeva_v2(z,13)+ \cite{Zaghloul16} &$4.07\times10^{-15}$  & $1.71\times 10^{-13}$&$11.00\,\,(\pm 0.11)$  \\
 \hline
\end{tabular}
\caption{Maximum absolute and relative errors for the Matlab codes implementing the approximations \eqref{wNdef}, \eqref{weid},  \eqref{abrar_approx3}, and \eqref{Zag_approx1}. The computing times are mean and s.d.\ of 25 executions.}
\label{table:w_compt_imes}
\end{table}

\vspace*{-1ex}
In Table \ref{table:w_compt_imes} we compare the accuracy and efficiency of our approximation and our (double precision) Matlab code  \verb+wTrap(z,N)+ with  (double precision) Matlab implementations of the approximations \eqref{weid}, \eqref{Zag_approx1}, and \eqref{abrar_approx3}. Results are shown in Table \ref{table:w_compt_imes} for:
\begin{enumerate}
  \item Our approximation $w_{N}(z)$ with $N=11$ implemented by the call \verb+wTrap(z,11)+ to the Matlab code provided in \cite[Table SM1]{ACW202};
  \item Weideman's approximation \eqref{weid} with $N=40$ (this choice of $N$ ensures high accuracy throughout the whole first quadrant, see \cite[Figure 8]{Weid94,Weid95}, \cite[Figure 2]{Mohd14}), implemented by the call \verb+cef(z,40)+ to the Matlab code in \cite[Table 1]{Weid94};
  \item The approximation of Abrarov et al., implemented by the call \verb+fadsamp(z)+ to the Matlab function in \cite[Appendix]{Abrar18}, which uses the method \eqref{abrar_approx3} with $\alpha=2.75$ and $M=23$ in the formulae for $\psi_{1}(z)$ and $\psi_{2}(z)$;
  \item The approximation \eqref{Zag_approx1} of Zaghloul and Ali \cite{Zag12}, implemented by the call \verb+Faddeyeva_v2(z,M)+ with $M=13$ (the maximum value permitted by the code) to the Matlab code described in \cite{Zaghloul16}: here $M$ is the number of accurate significant figures required.
\end{enumerate}
For these approximations Table \ref{table:w_compt_imes} shows estimated maximum absolute and relative errors in the first quadrant, and computation times (mean and standard deviation of 25 executions, each measured by Matlab \verb+timeit+) running Matlab version 9.3.0.713579 (R2017b) on a laptop with a single Intel64 Family 6 Model 78 2.40 GHz processor (and with \verb+maxNumCompThreads+ set to its default value of 2). The estimated maximum errors are discrete maxima over the same $1,602,801$ points as above, and we  again use as the exact value the approximation $w_{20}(z)$ given by  \eqref{wNdef}, implemented in quadruple precision by \verb+wTrap_Q(z,20)+. The computation times are for the case where \verb+z+ is a matrix containing the $16,008,001$ points $z=x+\ri y$, with $x=0(0.0025)10$, $y=0(0.0025)10$.

As measured by maximum absolute and relative errors over this large discrete set of $z$ values covering the first quadrant, our approximation $w_{11}(z)$ implemented as \verb+wTrap(z,11)+ is marginally the most accurate, though Weideman's approximation is essentially as accurate and all four methods achieve $<4\times 10^{-14}$ maximum absolute error, $<2\times 10^{-13}$ maximum relative error.

In these specific calculations our Matlab code is, on average over 25 realisations, marginally slower than that of Weideman, but faster than the other two. But little should be read into these timings, beyond, possibly,  that the first three of the methods are about equally fast, and the fourth a little slower. In particular we note that:
\begin{itemize}
\item[i)] Repeating these timings, even on exactly the same machine with exactly the same version of Matlab, will give slightly different results (see the additional timings results table in the supplementary materials).
\item[ii)] These comparisons are made with the specific Matlab implementations published by the various proposers of the algorithms. It may be that these timings could be reduced for any or all of these methods by better Matlab implementation strategies.
\item[iii)] Our timing comparisons in Table \ref{table:w_compt_imes} are for computing $w(z)$ on a uniformly distributed grid of points across the square $\{z=x+\ri y: 0\leq x\leq 10, 0\leq y\leq 10\}$. A different choice of points would lead to different results, at least for our approximation $w_N(z)$ and for the approximation of Abrarov et al., given by \eqref{abrar_approx3}, since both these approximations use somewhat different formulae in different regions of the first quadrant.
\end{itemize}

We leave to future publications more detailed comparisons of operation counts, timings, and accuracy of the above methods for computing $w(z)$ for complex $z$, along the lines of \cite{Weid94,Weid95} or \cite{Zag12,Zaghloul16}. Such publications might also study alternative, potentially more efficient implementations, for example using SIMD-aware C$++$ codes. We remark that C$++$ implementations of the method of Zaghloul and Ali \cite{Zag12} and of the continued fraction approximation \eqref{eq:nthconvergent} advocated by Gautschi \cite{WG70} (which is used in the \verb+fadsamp+ code of Abrarov et al.\ \cite{Abrar18}), are the basis of the widely used (though unpublished) Faddeeva Package of Johnson \cite{johnson_2021}.

\section*{Acknowledgments} We would like to acknowledge the detailed and helpful comments of the two anonymous referees whose input has led to significant improvements.


\end{document}